\documentclass[12pt,leqno]{amsart}
\usepackage{amsmath,amssymb,amsfonts}
\usepackage{eucal,enumerate}
\usepackage[all]{xy}
\usepackage{graphicx,psfrag}
\usepackage[utf8]{inputenc}	
\usepackage{mathrsfs}	

\newtheorem{theorem}{Theorem}
\newtheorem{lemma}[theorem]{Lemma}
\newtheorem{corollary}[theorem]{Corollary}

\theoremstyle{definition}
\newtheorem{construct}{Construction}

\newcommand\cA{\mathcal A}
\newcommand\cD{\mathcal D}
\newcommand\G\Gamma
\renewcommand\L\Lambda
\renewcommand\S\Sigma
\newcommand\s\sigma

\begin{document}

{\large
\section*{
Consistency in the Naturally Vertex-Signed Line Graph\\ of a Signed Graph
}}

\begin{center}
Thomas Zaslavsky
\\[10pt]
Department of Mathematical Sciences \\
Binghamton University (SUNY) \\
Binghamton, NY 13902-6000, U.S.A.
\\
Email:  {\tt zaslav@math.binghamton.edu}
\\[10pt]
\end{center}

\begin{center}
Dedicated to a great man, Dr.\ B.\ Devadas Acharya (1947--2013)
\end{center}

{\small
\begin{quote}
\textbf{Abstract.}  
A \emph{signed graph} is a graph whose edges are signed.  In a \emph{vertex-signed graph} the vertices are signed.  The latter is called \emph{consistent} if the product of signs in every circle is positive.  The line graph of a signed graph is naturally vertex-signed.  Based on a characterization by Acharya, Acharya, and Sinha in 2009, we give constructions for the signed simple graphs whose naturally vertex-signed line graph is consistent.

\textbf{Mathematics Subject Classification (2010)}:  Primary 05C22; Secondary 05C75.

\textbf{Keywords}:  Signed graph, line graph, consistent vertex-signed graph, consistent marked graph.
\end{quote}
}

\section{Introduction}

A \emph{signed graph} $\S=(\G,\s)$ consists of a graph $\G=(V,E)$, called its \emph{underlying graph}, and a sign function $\s: E\to\{+1,-1\}$.  The most essential question about a signed graph is whether it is \emph{balanced}, that is, the product of the edge signs in every circle (cycle, polygon, circuit) is positive.  Signed graphs were introduced by Harary \cite{NB}.  The vertex analog is a \emph{vertex-signed graph} (often called a \emph{marked graph}) in which the vertices are signed; the vertex analog of balance is \emph{consistency}, the property that in every circle the product of vertex signs is positive.  These analogs were introduced by Beineke and Harary \cite{BH}.  

As the edges of a simple graph $\G$ become the vertices of its line graph, if $\G$ is signed by $\s$ then $L(\G)$ naturally has its vertices signed by $\s$; it is a vertex-signed graph, which we call the \emph{naturally vertex-signed line graph} of $\S$ and denote by $\L_\s(\S)$.  The natural question is then to find a characterization of signed graphs whose naturally vertex-signed line graphs are consistent in the sense of Beineke and Harary.  (Then $\S$ may be called \emph{line consistent} \cite{CLC}.)  This question was taken up by Acharya, Acharya, and Sinha in 2009 \cite{AAS}; their solution was the following theorem, which they proved by means of the characterization of consistent vertex-signed graphs due to Hoede \cite{H}.  
(Slilaty and Zaslavsky \cite{CLC} simplify the theorem and give a short proof.)
The degree of a vertex is $d(v)$; the \emph{negative degree} $d^-(v)$ is the number of incident negative edges.

\begin{theorem}[{Characterization from \cite[Theorem 2.1]{AAS}}]\label{T:aas}
Assume $\S$ is a signed simple graph.  Then $\L_\s(\S)$ is consistent if and only if\/ $\S$ satisfies both the following conditions:
\begin{enumerate}[Property 1. ]
\item $\S$ is balanced.
\label{Prop1}
\item For each vertex $v\in V$, 
\label{Prop2}
\begin{enumerate}[{\rm(a)}]
\item if $d(v)>3$, then $d^-(v)=0$;
\item if $d(v)=3$, then $d^-(v)=0$ or $2$;
\item if $d^-(v)=2$ and $v$ lies on a circle, then both negative edges at $v$ belong to that circle.
\end{enumerate}
\end{enumerate}
\end{theorem} 

One interprets Property \ref{Prop2}(c) to mean that the two negative edges lie on every circle through $v$; thus (as observed in \cite{CLC}) the third edge at $v$ (if it exists) is positive and is an isthmus of $\S$.  

Theorem \ref{T:aas} is surprisingly simple when applied to signed blocks.

\begin{corollary}[\cite{BSG}]\label{C:2conn}
Let $\S$ be a signed simple, $2$-connected graph.  Then $\L_\s(\S)$ is consistent if and only if\/ $\S$ is balanced and all endpoints of negative edges have degree at most $2$.
\end{corollary}

In this paper we build upon Theorem \ref{T:aas} by providing constructive characterizations of signed graphs that satisfy Property \ref{Prop2} and thereby of those signed simple graphs that are line consistent.  

\section{Definitions}

We need definitions about graphs, some of which are unusual.  We denote an edge with endpoints $u$ and $v$ by $uv$, even if there exist other edges with the same endpoint.  (We can do that here because we do not treat parallel edges explicitly.)  
A loop contributes 2 to the degree of its vertex.  
A \emph{block} of a graph is a maximal subgraph $B$ such that any two edges in $B$ belong to a circle together.  (An isolated vertex, an isthmus, and a loop are blocks.)  A \emph{nontrivial block} is a block that contains a circle; if the graph is simple it is a block of order three or more.   A \emph{megablock} of a graph is a maximal connected union of one or more nontrivial blocks.  

A path may be open or closed (unlike the usual definition that excludes closed paths); it joins two \emph{termini}, which are equal in the case of a closed path.  An open path may have length 0; it is \emph{nontrivial} if it has positive length.  A closed path must have positive length.  The \emph{internal vertices} of a path are the vertices other than its termini.  The \emph{terminal edges} of a path are its edges that are incident with its termini.  
A circle is the graph of a closed path; the difference between a closed path and a circle is that a closed path has a terminus; a circle has no terminus and all its vertices are internal.

For a subgraph $\G' \subseteq \G$ define a \emph{$\G'$-divalent} path or circle to be a nontrivial path or circle in $\G'$ whose internal vertices have $d_{\G'}(v)=2$.  

In a signed graph the \emph{sign} of a path or circle is the product of the signs of its edges.  The \emph{negative subgraph} of $\S$ is the (unsigned) graph $\S^-$ that has all the vertices and all the negative edges of $\S$.  

By \emph{suppressing a divalent vertex} $v$ in a graph, we mean replacing $v$ and its incident edges $uv$, $vw$ by a single edge $uw$.  When we suppress a divalent vertex in a signed graph, we give the new edge $uw$ the sign $\s(uw) := \s(uv)\s(vw)$.  It is clear that, when suppressing several divalent vertices in a path or circle, the order in which we suppress them does not affect the result.  There is one kind of divalent vertex that cannot be suppressed: a divalent vertex that supports a loop.  By \emph{suppressing all possible divalent vertices} we mean suppressing divalent vertices until the only ones remaining are those that support loops.

Property \ref{Prop1}, balance, is well characterized by the first theorem of signed graph theory.  A \emph{cut} is the set of edges with one endpoint in some subset $X\subset V$ and the other endpoint in $V \setminus X$, provided there is at least one such edge (since the empty set is not a cut).

\begin{theorem}[Harary \cite{NB}]\label{T:balance}
A signed graph is balanced if and only if its set of negative edges is empty or a cut.
\end{theorem}

\section{Constructions}

We present four (more precisely, two and two halves) constructions that enforce Property \ref{Prop2}.  The first construction is the simplest.  The second has a smaller initial graph and reveals more about line-graph consistency.  The third is a variant of the second in which the signs are chosen late rather than early, and the fourth is a special case of the second in which balance is assured through the process of construction.

\begin{construct}\label{CA}
Let $\G$ be a graph.  Choose any set $\cA$ of pairwise disjoint paths and circles in $\G$ such that for each one, $P$, either 
\begin{enumerate}[(i)]
\item $P$ is a $B$-divalent path in a nontrivial block $B$ of $\G$, every vertex of $P$ has $d_\G(v)\leq3$, and its termini have $d_\G(v)=2$ 
(note that no vertex of $P$ can be a cutpoint of a megablock subgraph as any such cutpoint has degree at least 4; also, the third edge at each trivalent vertex in $P$ will be an isthmus of $\G$); or  
\item $P$ is a path consisting of isthmi, every vertex of $P$ has $d_\G(v)\leq3$, and its termini have $d_\G(v)\leq2$ 
(note that all edges adjacent to $P$ are isthmi; for if such an edge $uv$, with $u$ in $P$ but not an terminus, belonged to a nontrivial block $B$, $u$ would have degree at least 2 in $B$ and 2 in $P$, thereby having $d_\G(u)\geq4$; similarly, a terminus $u$ would have $d_\G(u)\geq3$); or 
\item $P$ is a megablock of $\G$ that is a circle and every vertex of $P$ has $d_\G(v)\leq3$ 
(note that any third edge will be an isthmus).
\end{enumerate}
Make the edges of the chosen paths and circles negative and the remaining edges positive.
\end{construct}

The purpose of Construction \ref{CB} is to clarify the structure of a line-consistent signed graph by starting with a smaller graph which is signed and enlarged so as to satisfy Property \ref{Prop2}.  

\begin{construct}\label{CB}
Let $\G'$ be a graph with no divalent vertices except those, if any, that support a loop.  
\emph{Subdividing} an edge means replacing it by a nontrivial path with the same termini or (if it is a loop) a circle; in particular, retaining the edge is considered a subdivision of that edge.  
We construct a graph $\G$ by subdividing $\G'$ and then we sign the edges of $\G$ to get $\S$.
\begin{enumerate}[Step 1.]
\item Sign $\G'$ arbitrarily.  Call this signed graph $\S'$.
\label{BS'}
\item Choose a subset $F'$ of edges in $\G'$ such that $F' \supseteq E^-(\G')$.
\label{BF'}
\item Partition the edges of $F'$ into nontrivial paths $P'$ and circles $C'$, so that 
  \begin{enumerate}[ (a)]
\item each path or circle has internal vertices of degree at most 3, 
\item the third edge at each trivalent internal vertex is an isthmus,  
\item each open path is either contained within a nontrivial block or composed entirely of isthmi (note that a closed path or a circle can only lie within a nontrivial block), and 
\item no terminus is divalent in $\G'$. 
  \end{enumerate}
Let $\cD'$ be the set of these paths and circles.  This partitioning can always be done; at worst, each edge is its own path, or circle if a loop component.  
(Note that (d) only forces a loop component in $F'$ to be a circle in $\cD'$.)
\label{BD'}
\item Subdivide each edge $e' \in E'$ into a path $P_{e'}$.  The resulting graph is $\G$.  Write $P$ for the path or circle that results from subdividing the edges of $P' \in \cD'$ and let $\cD$ be the set of paths and circles $P$ derived from all $P' \in \cD'$.
\label{Bsubdiv}
\item Sign $P_{e'}$ all positive if $e' \notin F'$.
\label{Bpos}
\item For each path or circle $P'\in\cD'$, sign the edges of $P$ so that for each edge $f'$ in $P'$, 
  \begin{enumerate}[ (a)]
\item $\s(P_{f'}) = \s'(f')$,
\item $P_{f'}$ is not all positive, 
\item a terminal edge of $P$ that is incident with a non-univalent terminus is positive, and 
\item any edge of $P$ that is incident to a trivalent internal vertex is negative. 
  \end{enumerate}
\label{Bpath}
\end{enumerate}
\end{construct}

It is sufficient in Step \ref{Bpath}(b) to assume that $P$ is not all positive.  (If some $P_{f'} \subset P$ were all positive, it would be impossible to satisfy Step \ref{Bpath}(d) at both termini of $P_{f'}$.)

Construction \ref{CC} is a variant of Construction \ref{CB} in which the signs of the initial graph $\G'$ are not specified until the end.

\begin{construct}\label{CC}
Carry out Construction \ref{CB} but omit Steps \ref{BS'} and \ref{Bpath}(a); also, in Step \ref{BF'} the choice of $F'$ is arbitrary.  After the construction of $\S$, form $\S'$ by defining $\s'(e'):=\s(P_{e'})$ for each edge $e' \in E'$.
(Note that since $\G$ in Construction \ref{CC} is the same as in Construction \ref{CB}, it follows that $\S'$ in Construction \ref{CC} is the same as in Construction \ref{CB}.)
\end{construct}

Construction \ref{CD} differs from \ref{CB} by inserting balance at the beginning; thus its constructs necessarily have Properties \ref{Prop1} and \ref{Prop2}, though they need not be simple graphs.

\begin{construct}\label{CD}
This is the same as Construction \ref{CB} except that in Step \ref{BS'}, each nontrivial block of $\G'$ is signed so it is balanced (but otherwise arbitrarily).  By Theorem \ref{T:balance}, to do that either make all edges in the block positive or choose a cut in the block and make it negative, the other edges being positive. 
\end{construct}

\section{Results}

Here is our theorem.

\begin{theorem}[Constructive Characterization]\label{T:cons}
Assume $\S$ is a signed graph, not necessarily simple.  
\begin{enumerate}[{\rm(a)}]
\item  The following properties of a signed graph $\S$ are equivalent:
\label{T:cons-1}
  \begin{enumerate}[{\rm(i)}]
\item  $\S$ satisfies Property \ref{Prop2}.
\item  $\S$ is constructed by Construction \ref{CA}.
\item  $\S$ is constructed by Construction \ref{CB}.
\item  $\S$ is constructed by Construction \ref{CC}.
  \end{enumerate}
\item  The following properties of a signed simple graph $\S$ are equivalent:
\label{T:cons-2}
  \begin{enumerate}[{\rm(i)}]
\item $\L_\s(\S)$ is consistent.
\item $\S$ is constructed by Construction \ref{CA} (or \ref{CB} or \ref{CC}) and is balanced.
\item $\S$ is constructed by Construction \ref{CD}.
  \end{enumerate}
\end{enumerate}
\end{theorem}

The proof will appear after some preliminary results.  

\begin{lemma}\label{T:bal}
{\rm(a)}  A signed graph $\S$ is balanced if and only if the signed graph resulting from it by suppressing all possible divalent vertices is balanced.

{\rm(b)}  In particular, a signed graph resulting from Construction \ref{CA} is balanced if and only if, in each nontrivial block $B$ of $\S$, when all possible divalent vertices (with degrees measured in $B$, not in $\S$) are suppressed, the negative edge set becomes empty or a cut.

{\rm(c)}  In Construction \ref{CD}, $\S$ is balanced.
\end{lemma}

\begin{proof}
For part (a), let $\S'$ be the result of suppressing all possible divalent vertices in $\S$.  Note that a $\S$-divalent path $P$ (in $\S$) and the single edge $e$ it becomes in $\S'$ have the same sign and belong to the same circles, if we identify circles in $\S$ with the circles in $\S'$ that result after suppression.  Those observations make the first assertion obvious.  Part (c) is a special case.

For (b), it is clear that $\S$, resulting from Construction \ref{CA}, is balanced if and only if each block $B$ is balanced.  A trivial block is always balanced.  Let $B'$ be the signed graph resulting from suppression applied to $B$.  By (a), $B$ is balanced if and only if $B'$ is balanced.  Balance of $B'$ is determined using Harary's theorem.
\end{proof}

Note that the suppression in Lemma \ref{T:bal}(a) can produce a graph that is not simple even if $\S$ is simple.

\begin{lemma}\label{L:prop3}
Property \ref{Prop2} of a signed graph $\S$ is equivalent to:  
\begin{enumerate}[Property 1. ] 
\setcounter{enumi}{2}
\item For each vertex $v \in V$, 
\label{Prop3}
\begin{enumerate}[{\rm(a)}]
\item $d^-(v) = 0$, or
\item $1 \leq d^-(v) \leq d(v) \leq 2$, or
\item $d^-(v) = 2$, $d(v) = 3$, and the positive edge at $v$ is an isthmus.
\end{enumerate}
\end{enumerate}
\end{lemma}

\begin{proof}
Considering all possible cases for $v$ in Property \ref{Prop2}, $d^-(v)$ is at most 2 and when $d^-(v)>0$ then $d(v) \leq d^-(v)+1$.  If $d^-(v)=2$, Property \ref{Prop2}(c) entails that a third edge at $v$ must be an isthmus.

It is easy to check that Property \ref{Prop3} implies each part of Property \ref{Prop2}.
\end{proof}

\begin{lemma}[Uniqueness]\label{L:unique}
The initial graph $\G'$ in Construction \ref{CB} is determined (up to isomorphism) by the resulting unsigned graph $\G$.  

A signed graph can be constructed by Construction \ref{CB} in only one way.  That is, $\S'$, $F'$, and the list $\cD'$ are determined by the resulting signed graph $\S$.
\end{lemma}

\begin{proof}
Both $\G'$ and $\S'$ are obtained by suppressing all possible divalent vertices in $\G$ and $\S$, respectively.  
Similarly, $F'$ consists of all edges of $\S'$ that after subdivision are not all positive.  

To see how those edges associate to form $\cD'$, consider distinct edges $e$ and $f$ in $\S$ incident to a vertex $v$ of $\S'$.  
Let $e'$ and $f'$ be the edges of $\S'$ such that $e \in P_{e'}$ and $f \in P_{f'}$ and assume first that $e' \neq f'$, so $d_\S(v)\geq3$.  The only way $e$ and $f$ can be in the same path or circle of $\cD$ is when both are negative; and if that is the case, then neither can be a terminal edge so they must be in the same path or circle.  
Now assume that $e'=f'$; then that edge is a loop in $\S'$ and $P'$ is a closed path or a circle consisting of that loop and its supporting vertex $v$.

If the graph of $P$ is a circle, we need to decide whether $P$ is a closed path or a circle.  It must be a circle if it has only divalent vertices.  Otherwise, only a vertex of $P$ that is trivalent and at which the incident edges of $P$ are both positive can be a terminus. Thus, $P$ is a closed path or a circle according as such a vertex in $P$ exists or does not.
  
Since the elements of $\cD'$ are determined by $\S$, the proposition is proved.
\end{proof}

\begin{proof}[Proof of Theorem \ref{T:cons}\eqref{T:cons-1}]
We may substitute Property \ref{Prop3} for Property \ref{Prop2} in (i).

(i)$\implies$(ii):  Suppose $\S$ is a signed graph that satisfies Property \ref{Prop3}.  
Since $d^-(v)\leq2$ for every vertex, the components of $\S^-$ are paths (possibly trivial) and circles.

In a circle component $P$ of $\S^-$, every vertex has $d_\G(v)\leq3$ and, by Property \ref{Prop3}(c), if there is a third edge at $v$, it is an isthmus.  Hence $P$ is a megablock, so it is an instance of Construction \ref{CA}(iii). 

Take a path or circle component $P$ of $\S^-$ that has an edge $f$ in a nontrivial block $B$.  Consider (if one exists) an edge $g$ of $P$ that is adjacent to $f$, say at vertex $v$.  Since $d_\G(v)\leq3$ by Property \ref{Prop3}(b, c), either $d_\G(v)=2$ and both edges at $v$ are in $B$, or $d_\G(v)=3$.  In the latter case the third edge, $h$, is an isthmus by Property \ref{Prop3}(c).  Then $g$ cannot be an isthmus; it must be in $B$.  It follows that all edges of $P$ are in $B$ and $P$ is $B$-divalent.  The termini of $P$ have $d_\G(v)\leq2$ by Property \ref{Prop3}(b), and they must be divalent because no edge of $P$ is an isthmus.  That is, $P$ satisfies Construction \ref{CA}(i).

Finally, consider $P$ that is a path consisting of isthmi.  Its internal vertices have $d_\G(v)\leq3$, and by Property \ref{Prop3}(b) its termini have $d_\G(v)\leq2$.  Hence, it satisfies Construction \ref{CA}(ii).

(ii)$\implies$(iii):  Let $\S_A$ result from Construction \ref{CA}.  To obtain it from Construction \ref{CB} we begin with $\S'$ and $\G'$ obtained by suppressing all possible divalent vertices in $\S_A$ and its underlying graph $\G_A$.  By Lemma \ref{L:unique} those are the proper initial graph and signed graph to get $\S_A$ from Construction \ref{CB}.  As in the proof of that lemma, $F'$ is the set of edges $e'\in E'$ such that $P_{e'}$ is not all positive.  We also define $\cD$ and $\cD'$ as in the proof of Lemma \ref{L:unique}; that is possible because the terminus of a negative path $N \in \cA$ cannot have degree greater than 2, so if it has degree 2 the negative path $N$ can be continued with a positive edge before a vertex of degree greater than 2 is reached.  We need only to verify that $\cD'$ satisfies the assumptions of Construction \ref{CB}.

Clearly, $\cD'$ partitions $F'$ into nontrivial paths and circles.  Let $P'$ be one such path or circle, corresponding to $P \in \cD$, and let $v$ be an internal vertex of $P'$.  The edges $e, f \in P$ incident with $v$ are both negative, by the construction of $\cD$.  Hence, they belong to a path $P_A$ in Construction \ref{CA} and therefore $d_{\G'}(v)=d_{\G_A}(v)\leq3$ and a third edge, if there is one, is an isthmus of $\G_A$, thus of $\G'$.  Moreover, $e$ and $f$ are both isthmi or both contained within a nontrivial block of $\G_A$.  Finally, no terminus is divalent in $\G'$ by the construction of $\G'$ and $\cD'$.  Therefore, the requirements of Step \ref{BD'} are satisfied.

The requirements of Step \ref{Bpos} and Step \ref{Bpath}(a, b) are satisfied by the definition of $F'$.  Those of Step \ref{Bpath}(c, d) is satisfied due to the way we constructed $\cD$.  

It follows that the signed graph $\S$ resulting from Construction \ref{CB} given the initial data can be made to agree with $\S_A$; i.e., $\S_A$ is constructible by Construction \ref{CB}.

(iii)$\implies$(i):  Let $\S$ result from Construction \ref{CB}.  We show it has Property \ref{Prop3}.  A \emph{new vertex} of $\S$ is one that results from subdividing edges of $\S'$; any other vertex is an \emph{original vertex}.

Suppose $e$ is a negative edge in some path or circle $P\in\cD$ and is incident with vertex $v$.  If $d_\G(v)\leq2$, $v$ satisfies Property \ref{Prop3}.  Suppose, therefore, that $d_\G(v)\geq3$.  Since $e$ is negative, by Step \ref{Bpath}(c) $v$ cannot be a terminus of $P$.  Hence $d_\G(v)\leq3$ by Step \ref{BD'}(a) and the third edge $f$ incident with $v$ is an isthmus by Step \ref{BD'}(b).  The two edges of $P$ at $v$ are negative by Step \ref{Bpath}(d), so $d^-(v)\geq2$.  The edge $f$ is either terminal in a path $P_1\in\cD$, hence positive by Step \ref{Bpath}(c), or in a subdivision path $P_{f'}$ of an edge $f'\notin F'$, hence positive by Step \ref{Bpos}.  Either way, $d_\S^-(v)\leq2$.  Thus, $\S$ satisfies Property \ref{Prop3}.

(iii)$\iff$(iv):  The note in Construction \ref{CC} explains why Constructions \ref{CB} and \ref{CC} yield the same signed graphs.
\end{proof}

\begin{proof}[Proof of Theorem \ref{T:cons}\eqref{T:cons-2}]
The equivalence of (ii) and (iii) follows from \eqref{T:cons-1}[(ii)$\iff$(iii)] and the fact that $\S$ resulting from Construction \ref{CD} is balanced by Lemma \ref{T:bal}(c).

As for (i), it is equivalent to $\S$'s having Property \ref{Prop2} and being balanced, which is equivalent by \eqref{T:cons-1}[(i)$\iff$(ii)] to $\S$'s resulting from Construction \ref{CA} and being balanced, i.e., (ii).
\end{proof}



\begin{thebibliography}{99}

\bibitem{AAS} B.\ Devadas Acharya, Mukti Acharya, and Deepa Sinha,
Characterization of a signed graph whose signed line graph is $S$-consistent.  
\emph{Bull.\ Malaysian Math.\ Sci.\ Soc.}\ (2) {\bf 32} (2009), no.\ 3, 335--341.
MR 2562172 (2010m:\-05135).  Zbl 1176.05032.

\bibitem{BH} Lowell W.\ Beineke and Frank Harary, 
Consistent graphs with signed points.  
\emph{Riv.\ Mat.\ Sci.\ Econom.\ Social.}\ {\bf 1} (1978), 81--88.  
MR 81h:05108.  Zbl 493.05053.

\bibitem{NB} F.\ Harary, 
On the notion of balance of a signed graph. 
\emph{Michigan Math.\ J.}\ {\bf 2} (1953--54), 143--146 and addendum preceding p.\ 1.
MR 16, 733h.  Zbl 56, 421c (e: 056.42103).

\bibitem{H} Cornelis Hoede, 
A characterization of consistent marked graphs.  
\emph{J.\ Graph Theory} {\bf 16} (1992), 17--23.  
MR 93b:05141.  Zbl 748.05081.

\bibitem{CLC} Daniel C.\ Slilaty and Thomas Zaslavsky, 
Characterization of line-consistent signed graphs.  
Submitted.  

\bibitem{BSG} Thomas Zaslavsky, 
A mathematical bibliography of signed and gain graphs and allied areas.  
\emph{Electronic J.\ Combin.}, Dynamic Surveys in Combinatorics (1998), No.\ DS8 (electronic).
Eighth ed.\ (2012).

\end{thebibliography}
\end{document}